\documentclass[11pt]{amsart}
\usepackage[margin=1in]{geometry}

\usepackage{amssymb}
\usepackage{amsthm}
\usepackage{amsmath}
\usepackage{mathrsfs}
\usepackage{amsbsy}
\usepackage{bm}
\usepackage{hyperref}
\usepackage{tikz}
\usepackage{array}
\usepackage{enumerate}
\usepackage{xcolor}
\usepackage{bbm}
\usepackage{comment}
\usepackage{mathtools}

\definecolor{indigo}{rgb}{0.6,0,1}
\definecolor{lavender}{rgb}{1,0,1}
\definecolor{DarkGreen}{rgb}{0,0.75,0}
\definecolor{DarkBlue}{rgb}{0,0,0.8}
\definecolor{Navy}{rgb}{0,0,0.314}
\definecolor{Yellow}{rgb}{0.95,0.75,0}
\definecolor{Cyan}{rgb}{0,1,1}
\definecolor{Pink}{rgb}{1,0.431,1}

\usepackage[noabbrev,capitalise,nameinlink]{cleveref}

\hypersetup{colorlinks=true, citecolor=Yellow, linkcolor=DarkBlue,urlcolor=Yellow}

\usepackage{hhline}
\allowdisplaybreaks
\usepackage[noadjust]{cite}

\usepackage{caption}
\usepackage[noabbrev,capitalise,nameinlink]{cleveref}
\crefname{conjecture}{Conjecture}{Conjectures}

\newtheorem{theorem}{Theorem}[section]
\newtheorem{proposition}[theorem]{Proposition}

\newtheorem{question}[theorem]{Question}
\newtheorem{problem}[theorem]{Problem}
\newtheorem{lemma}[theorem]{Lemma}

\theoremstyle{definition}

\newcommand{\id}{\mathrm{id}}

\DeclareMathOperator{\Var}{Var}
\newcommand{\bpi}{\boldsymbol{\pi}}
\newcommand{\wt}{\mathrm{wt}}

\begin{document}

\title[]{Rainbow Stackings of Random Edge-Colorings}
\subjclass[2010]{}

\author[]{Noga Alon}
\address{Department of Mathematics, Princeton University, Princeton, NJ 08544, USA and Schools of Mathematics and Computer Science, Tel Aviv University, Tel Aviv 69978, Israel}
\email{nalon@math.princeton.edu}

\author[]{Colin Defant}
\address[]{Department of Mathematics, Harvard University, Cambridge, MA 02138, USA}
\email{colindefant@gmail.com}

\author[]{Noah Kravitz}
\address[]{Department of Mathematics, Princeton University, Princeton, NJ 08540, USA}
\email{nkravitz@princeton.edu}

\begin{abstract}
A \emph{rainbow stacking} of $r$-edge-colorings $\chi_1, \ldots, \chi_m$ of the complete graph on $n$ vertices is a way of superimposing $\chi_1, \ldots, \chi_m$ so that no edges of the same color are superimposed on each other.  We determine a sharp threshold for $r$ (as a function of $m$ and $n$) governing the existence and nonexistence of rainbow stackings of random $r$-edge-colorings $\chi_1,\ldots,\chi_m$. 
\end{abstract} 

\maketitle

\section{Introduction}\label{sec:intro}

Let $\mathfrak S_n$ denote the symmetric group of permutations of the set $[n]\coloneq\{1,\ldots,n\}$. Let $K_n$ denote the complete graph with vertex set $[n]$ and edge set $\binom{[n]}{2}$. Consider a set $\mathcal C_r$ of $r$ colors, and let $\chi_1,\ldots,\chi_m\colon\binom{[n]}{2}\to\mathcal C_r$ be edge-colorings of $K_n$. A \emph{{\color{red}r}{\color{orange}a}{\color{yellow}i}{\color{DarkGreen}n}{\color{DarkBlue}b}{\color{indigo}o}{\color{lavender}w} {\color{red}s}{\color{orange}t}{\color{yellow}a}{\color{DarkGreen}c}{\color{DarkBlue}k}{\color{indigo}i}{\color{lavender}n}{\color{red}g}} of $\chi_1,\ldots,\chi_m$ is a tuple ${\boldsymbol{\sigma}=(\sigma_1,\ldots,\sigma_m)\in\mathfrak S_n^m}$ such that for each edge $e\in \binom{[n]}{2}$, the colors $$\chi_1(\sigma_1^{-1}(e)),\ldots,\chi_m(\sigma_m^{-1}(e))$$ are all distinct (where $\sigma_k^{-1} (\{i,j\})\coloneq\{\sigma_k^{-1}(i),\sigma_k^{-1}(j)\}$). Less formally, a rainbow stacking is a way of stacking copies of $K_n$ with the colorings $\chi_1, \ldots, \chi_m$ on top of each other so that no edge is stacked above another edge of the same color (see \cref{fig:rainbow}).  

\begin{figure}[]
\begin{center}{\includegraphics[height=6.784cm]{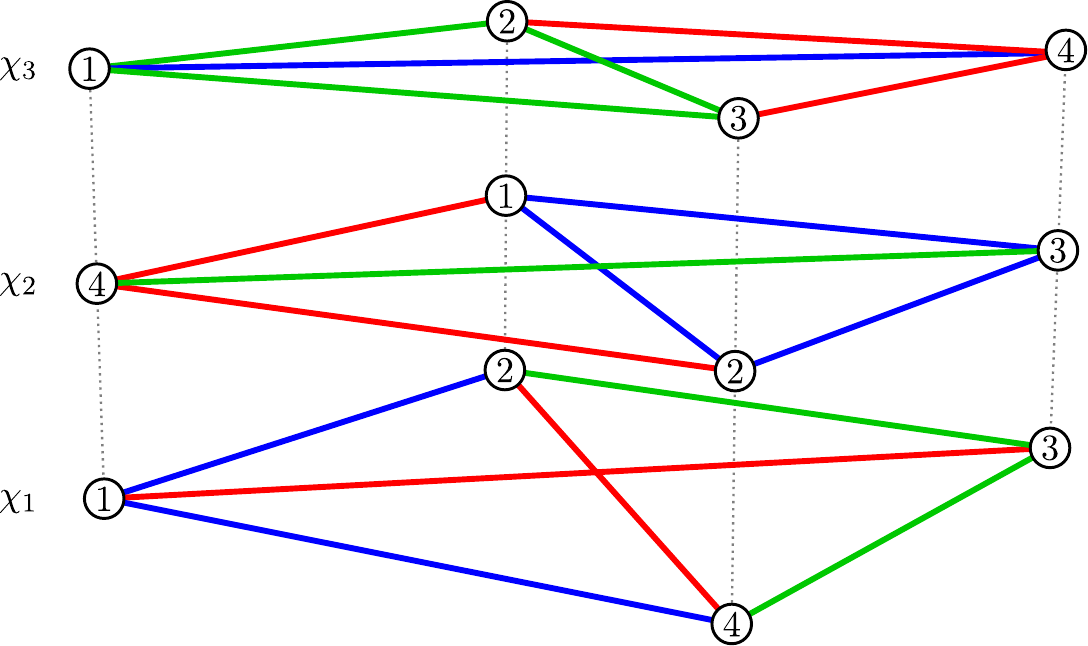}}
\end{center}
\caption{A rainbow stacking $\bpi\in\mathfrak S_4^3$ of $3$ edge-colorings of $K_4$, where there are $r=3$ total colors. The permutations in the tuple $\bpi$ are (in one-line notation) $\pi_1=1234$, $\pi_2=2431$, and $\pi_3=1243$.}
\label{fig:rainbow}
\end{figure}

We are interested in the existence of rainbow stackings, especially when $\chi_1,\ldots,\chi_m\colon \binom{[n]}{2} \to \mathcal C_r$ are independent uniformly random colorings.  When $m$ is fixed and $n$ is growing, we wish to determine which values of $r$ (in terms of $n$) guarantee the existence or nonexistence of rainbow stackings. In what follows, the phrase \emph{with high probability} always means \emph{with probability tending to $1$ as $n\to\infty$}. Our main theorem determines a sharp threshold that governs whether rainbow stackings exist with high probability or do not exist with high probability.  

Using the first-moment method, one can quickly find a upper bound on $r$ that guarantees the nonexistence of a rainbow stacking with high probability, as follows.  Let $\chi_1,\ldots,\chi_m\colon\binom{[n]}{2}\to\mathcal C_r$ be independent uniformly random edge-colorings.   For each $\boldsymbol{\sigma} \in \mathfrak S_n^m$, let $Z_{\boldsymbol{\sigma}}$ be the indicator function for the event that $\boldsymbol{\sigma}$ is a rainbow stacking of $\chi_1, \ldots, \chi_m$, and let $Z\coloneq\sum_{\boldsymbol{\sigma}} Z_{\boldsymbol{\sigma}}$ be the total number of rainbow stackings of $\chi_1, \ldots, \chi_m$. Note that $Z$ is always a multiple of $n!$; indeed, if $(\sigma_1, \ldots, \sigma_m)$ is a rainbow stacking of $\chi_1,\ldots,\chi_m$, then so is each tuple of the form $(\tau\sigma_1, \ldots, \tau\sigma_m)$ for $\tau \in \mathfrak{S}_n$. For each $\boldsymbol{\sigma} \in \mathfrak{S}_n^m$, the expectation of $Z_{\boldsymbol{\sigma}}$ is exactly
$$E_{n,m,r}\coloneq\prod_{i=1}^{m-1}\left(1-\frac{i}{r}\right)^{\binom{n}{2}}.$$
Consequently,
\begin{equation}\label{eq:first-moment-expression}
\mathbb{E}[Z] =n!^m E_{n,m,r}\leq n!^m \left(\prod_{i=1}^{m-1} e^{-i/r} \right)^{\binom{n}{2}}=n!\exp\left((m-1) \log(n!)-\binom{m}{2}\binom{n}{2}\cdot\frac{1}{r} \right).
\end{equation}
If there is a function $\omega\colon\mathbb N\to\mathbb R$ such that $\lim_{n\to\infty}\omega(n)=\infty$ and 
\[r\leq\frac{m\binom{n}{2}}{2\log(n!)}-\frac{\omega(n)}{(\log n)^2},\] then 
$$m\binom{n}{2}\cdot \frac{1}{r}-2\log(n!) \to \infty,$$
so $\mathbb{E}[Z]=o_m(n!)$ as $n \to \infty$. In this case, since the value of $Z$ is always a nonnegative integer multiple of $n!$, Markov's Inequality implies that
$$\mathbb{P}[Z>0]=\mathbb{P}[Z \geq n!] \leq \mathbb{E}[Z]/n!=o_m(1),$$
so with high probability, there are no rainbow stackings of $\chi_1, \ldots, \chi_m$.  This establishes the first half of the following theorem; the proof of the second half is the main work of this paper. 

\begin{theorem}\label{thm:main}
Fix an integer $m\geq 2$ and a function $\omega\colon \mathbb{N} \to \mathbb{R}$ such that $\lim_{n \to \infty} \omega(n)=\infty$.  For each $n\geq 1$, let $\chi_1,\ldots, \chi_m\colon \binom{[n]}{2} \to \mathcal{C}_r$ be independent uniformly random $r$-edge-colorings. If $r=r(n) \geq 1$ satisfies
\begin{equation}\label{eq:main1}
r \leq \frac{m\binom{n}{2}}{ 2\log(n!)}-\frac{\omega(n)}{(\log n)^2},
\end{equation}
then with high probability, there does not exist a rainbow stacking of $\chi_1,\ldots,\chi_m$. If 
\begin{equation}\label{eq:main2}
r \geq \frac{m\binom{n}{2}}{2\log(n!)}+\frac{2m-1}{3}+\frac{m}{2\log n}+\frac{\omega(n)}{(\log n)^2}, 
\end{equation}
then with high probability, there exists a rainbow stacking of $\chi_1,\ldots,\chi_m$. 
\end{theorem}

\section{Existence of Rainbow Stackings}
We will use the second-moment method to prove the second statement of \cref{thm:main}.  We already computed the first moment of $Z$ in \cref{sec:intro}.  The second moment of $Z$ is
$$\mathbb{E}[Z^2]=\mathbb{E}\left[ \left(\sum_{\boldsymbol{\sigma} \in \mathfrak{S}_n^m} Z_{\boldsymbol{\sigma}} \right)^2 \right]=\sum_{\boldsymbol{\sigma},\boldsymbol{\tau} \in \mathfrak{S}_n^m} \mathbb{E}[Z_{\boldsymbol{\sigma}}Z_{\boldsymbol{\tau}}].$$
For each $k \in [m]$, define the new coloring $\chi'_k:\binom{[n]}{2} \to \mathcal{C}_r$ by $\chi'_k(e)\coloneq\chi_k(\sigma_k^{-1}(e))$.   Now, $Z_{\boldsymbol{\sigma}}Z_{\boldsymbol{\tau}}$ is the indicator function of the event that for each $e \in \binom{[n]}{2}$, the colors $$\chi'_1(e), \ldots, \chi'_m(e)$$ are all distinct and the colors
$$\chi'_1(\sigma_1\tau_1^{-1}(e)), \ldots, \chi'_m(\sigma_m\tau_m^{-1} (e))$$
are all distinct.  Hence, $Z_{\boldsymbol{\sigma}}Z_{\boldsymbol{\tau}}$ has the same distribution as $Z_{\boldsymbol{\id}}Z_{(\sigma_1\tau_1^{-1}, \ldots, \sigma_m\tau_m^{-1})}$, where we write ${\boldsymbol{\id}=(\id,\ldots,\id)}$ for the tuple in $\mathfrak S_n^m$ whose components are all equal to the identity permutation $\id\in\mathfrak S_n$. Consequently,
$$\mathbb{E}[Z^2]=n!^m\sum_{\bpi \in \mathfrak{S}_n^m} \mathbb{E}[Z_{\boldsymbol{\id}} Z_{\bpi}].$$
We derive an explicit formula for each $[Z_{\boldsymbol{\id}} Z_{\bpi}]$ as follows.

For each $e\in\binom{[n]}{2}$, let $\beta_1(e),\ldots,\beta_m(e)$ be $m$ copies of $e$. Consider the $m$-partite graph $G_{\bpi}$ with vertex set $V(G_{\bpi})=\{\beta_k(e):e\in\binom{[n]}{2},\,k\in[m]\}$ in which $\beta_k(e)$ is adjacent to $\beta_{k'}(e')$ if and only if $k\neq k'$ and either $e=e'$ or $\pi_k(e)=\pi_{k'}(e')$.  The edge-colorings $\chi_1,\ldots,\chi_m$ naturally induce a vertex-coloring of $G_{\bpi}$, where $\beta_k(e)$ is assigned the color $\chi_k(e)$. Observe that $\boldsymbol{\id}$ and $\bpi$ are both rainbow stackings of $\chi_1, \ldots, \chi_m$ if and only if the induced vertex-coloring of $G_{\bpi}$ is proper. For example, if $\bpi$ and $\chi_1,\chi_2,\chi_3$ are as depicted in \cref{fig:rainbow}, then $G_{\bpi}$ and its induced coloring are shown in \cref{fig:Gpi}. Although $\bpi$ is a rainbow stacking of $\chi_1,\chi_2,\chi_3$, the identity tuple $\boldsymbol{\id}$ is not; this is why there are pink edges in \cref{fig:Gpi} whose endpoints have the same color. 

The quantity $\mathbb{E}[Z_{\boldsymbol{\id}} Z_{\bpi}]$ is equal to $r^{-m\binom{n}{2}}N_{\bpi}$, where $N_{\bpi}$ is the number of proper $r$-vertex-colorings of $G_{\bpi}$. Hence, we will study how $N_{\bpi}$ depends on $\bpi$. 

\begin{figure}[]
\begin{center}{\includegraphics[height=12.958cm]{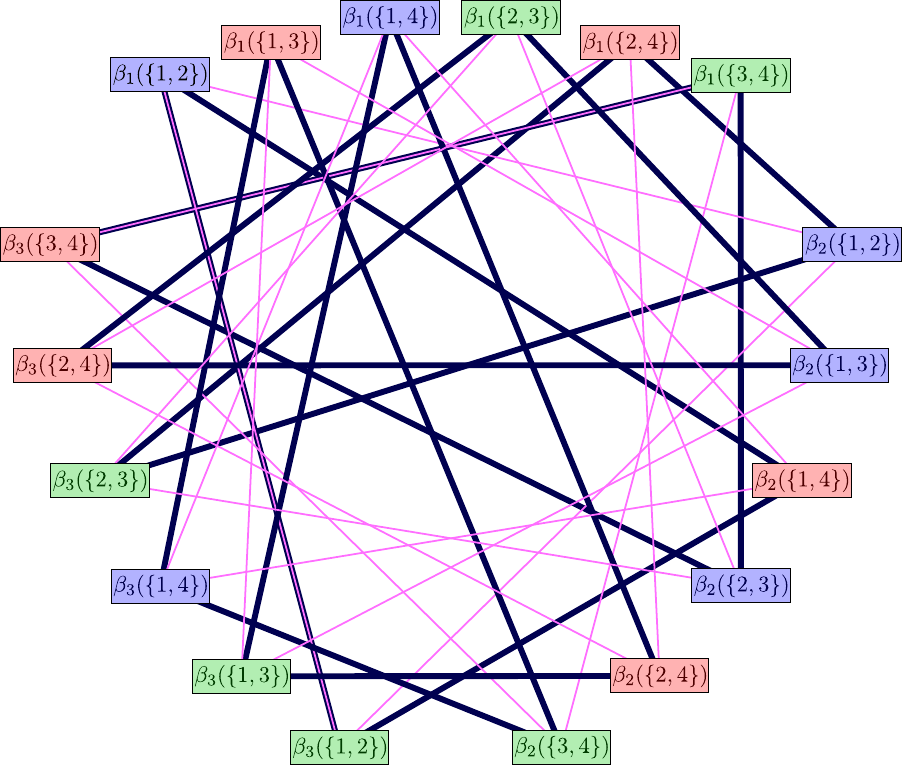}}
\end{center}
\caption{The graph $G_{\bpi}$, where $\bpi=(1234,2431,1243)$ is as depicted in \cref{fig:rainbow}. Edges of $G_{\bpi}$ of the form $\{\beta_k(e),\beta_{k'}(e)\}$ are drawn in {\color{Pink}thin pink}, while those of the form $\{\beta_k(e),\beta_{k'}(e')\}$ with $\pi_k(e)=\pi_{k'}(e')$ are drawn in {\color{Navy}\textbf{thick navy}}. Edges of the form $\{\beta_k(e),\beta_{k'}(e)\}$ with $\pi_k(e)=\pi_{k'}(e)$ are drawn with \includegraphics[height=0.2802cm]{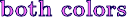}.  The vertex-coloring of $G_{\bpi}$ is induced by the edge-colorings $\chi_1,\chi_2,\chi_3$ in \cref{fig:rainbow}.}
\label{fig:Gpi}
\end{figure}

The graph $G_{\bpi}$ has $m \binom{n}{2}$ vertices and $$2\binom{m}{2}\binom{n}{2}-\left|\left\{(k,k',e) \in [m] \times [m] \times \binom{[n]}{2}: k<k' \text{ and } \pi_k(e)=\pi_{k'}(e) \right\}\right|$$
edges.  For each $k<k'$, we have
$$\left|\left\{e \in \binom{[n]}{2}:\pi_k(e)=\pi_{k'}(e)\right\}\right|=\binom{f_{k,k'}(\bpi)}{2}+t_{k,k'}(\bpi),$$
where $f_{k,k'}(\bpi)$ and $t_{k,k'}(\bpi)$ denote the number of fixed points and the number of $2$-cycles (respectively) of $\pi_k^{-1} \pi_{k'}$ (viewed as a permutation of $[n]$).  Define the weight $$\wt(\bpi)\coloneq\sum_{1 \leq k<k' \leq m}\wt_{k,k'}(\bpi),$$ where 
\[\wt_{k,k'}(\bpi)\coloneq\binom{f_{k,k'}(\bpi)}{2}+t_{k,k'}(\bpi).\]
Then the number of edges of $G_{\bpi}$ is
$$2\binom{m}{2}\binom{n}{2}-\wt(\bpi).$$
The following lemma provides an upper bound on $N_{\bpi}$ in terms of the weight $\wt(\bpi)$.

\begin{lemma}\label{lem:entropy-bound}
Let $m \geq 2$ be an integer. If $\bpi \in \mathfrak{S}_n^m$ and $r=r(n)$ satisfies $n^2/r^3=o(1)$, then
\[N_{\bpi} \leq (1+o_m(1))r^{m\binom{n}{2}} E_{n,m,r}^2 e^{\wt(\bpi)/(r-(2m-1)/3)}.\] 
\end{lemma}

\begin{proof}
We use the entropy method.  Let $\chi: V(G_{\bpi}) \to \mathcal{C}_r$ be a uniformly random proper $r$-coloring.  Then the entropy of $\chi$ is
$$H(\chi)=\log(N_{\bpi}),$$
where $H(\cdot)$ denotes the base-$e$ entropy function.  Let $\sigma \in \mathfrak{S}_m$ be a permutation.  We will reveal the values of $\chi$ on the vertices $\beta_{\sigma(1)}(e)$, then the vertices $\beta_{\sigma(2)}(e)$, and so on until the vertices $\beta_{\sigma(m)}(e)$.  For each stage, let $\chi^\sigma_{<k}$ denote the partial coloring on the vertices $\beta_{\sigma(k')}(e)$ for $k'<k$ and $e \in \binom{[n]}{2}$.  Then the chain rule and the subadditivity of entropy give that
$$H(\chi) \leq \sum_{k=1}^m \sum_{e \in \binom{[n]}{2}} H(\chi(\beta_{\sigma(k)}(e))\mid \chi^\sigma_{<k}).$$
We will estimate the summands appearing on the right-hand side of this inequality individually.

For each color $c\in\mathcal C_r$, each vertex $\beta_\ell(e)$, and each partial proper coloring $\chi'$ of the other vertices of $G_{\bpi}$, we have 
\begin{equation}\label{eq:claim}
\mathbb P[\chi(\beta_\ell(e))=c\mid \text{$\chi$ and $\chi'$ agree wherever $\chi'$ is defined}] \leq 1/(r-2m+2).  
\end{equation}
Indeed, because $\beta_\ell(e)$ has at most $2m-2$ neighbors, there are at most $2m-2$ forbidden values of $\chi(\beta_\ell(e))$; since the remaining colors are equally likely, each one occurs with probability at most $1/(r-2m+2)$.

Now, consider a single permutation $\sigma$ and a single vertex $\beta_{\sigma(k)}(e)$.  Let $y=y^\sigma_k(e)$ be such that $2(k-1)-y$ is the number of distinct colors already assigned by $\chi$ to the neighbors of $\beta_{\sigma(k)}(e)$ that are of the form $\beta_{\sigma(k')}(e')$ with $k'<k$.  Then there are at most $r-2(k-1)+y$ possibilities for $\chi(\beta_{\sigma(k)}(e))$, and the entropy of $\chi(\beta_{\sigma(k)}(e))$ conditioned on this partial coloring is at most
$$\log(r-2(k-1)+y)=\log(r-2(k-1))+\log \left(1+\frac{y}{r-2(k-1)} \right)\leq\log(r-2(k-1))+\frac{y}{r-2(k-1)}.$$
Summing over all of the possibilities for the partial coloring $\chi^\sigma_{<k}$,
we find that
\begin{equation}\label{eq:entropy-in-terms-of-y}
H(\chi(\beta_{\sigma(k)}(e))\mid \chi^\sigma_{<k}) \leq \log(r-2(k-1))+\mathbb{E}[y]\frac{1}{r-2(k-1)}.
\end{equation}
The next task is estimating $\mathbb{E}[y]$.

For each triple $(\ell,\ell',e)$ with $\ell< \ell'$ and $e \in \binom{[n]}{2}$, let \[x(\ell,\ell',e)=\begin{cases} 1 & \mbox{if }\pi_\ell(e)=\pi_{\ell'}(e); \\   0 & \mbox{otherwise}.\end{cases}\]  
We record for future reference the identity
\begin{equation}\label{eq:-sum-of-x's}
\sum_{\ell<\ell'} \sum_{e \in \binom{[n]}{2}} x(\ell,\ell',e)=\wt(\bpi).
\end{equation}
The neighbors of $\beta_{\sigma(k)}(e)$ already colored by $\chi^\sigma_{<k}$ are the vertices $\beta_{\sigma(k')}(e)$ and $\beta_{\sigma(k')}(\pi_{\sigma(k')}^{-1}\pi_{\sigma(k)}(e))$ for $k'<k$.  Counting collisions, we find that the number of such vertices is $$2(k-1)-\sum_{k'<k}x(\sigma(k),\sigma(k'),e).$$ 
The vertices $\beta_{\sigma(k')}(e)$ for $k'<k$ form a clique in $G_{\bpi}$; likewise, the vertices $\beta_{\sigma(k')}(\pi_{\sigma(k')}^{-1}\pi_{\sigma(k)}(e))$ for $k'<k$ form a clique in $G_{\bpi}$.  So the pairs of such vertices receiving the same color form a matching, and the number of such pairs is at least $y-\sum_{k'<k}x(\sigma(k),\sigma(k'),e)$.  Each pair of vertices $\beta_{\sigma(k'')}(e), \beta_{\sigma(k')}(\pi_{\sigma(k')}^{-1}\pi_{\sigma(k)}(e))$, for ${k',k''<k}$, receives the same color with probability at most $1/(r-2m+2)$ by \eqref{eq:claim}, so 
\begin{align*}
\mathbb{E}[y] &\leq \frac{(k-1-\sum_{k'<k}x(\sigma(k),\sigma(k'),e))^2}{r-2m+2}+\sum_{k'<k}x(\sigma(k),\sigma(k'),e)\\
 &\leq \frac{(k-1)(k-1-\sum_{k'<k}x(\sigma(k),\sigma(k'),e))}{r-2m+2}+\sum_{k'<k}x(\sigma(k),\sigma(k'),e)\\
&= \frac{(k-1)^2}{r-2m+2}+\left(1-\frac{k-1}{r-2m+2} \right)\sum_{k'<k}x(\sigma(k),\sigma(k'),e)\\
 &=\frac{(k-1)^2}{r-2m+2}+\left(\frac{r-2(k-1)}{r-k+1}+O_m(1/r^2) \right)\sum_{k'<k}x(\sigma(k),\sigma(k'),e).
\end{align*}
Substituting this into \eqref{eq:entropy-in-terms-of-y}, summing over $k$ and $e$, and averaging over $\sigma \in \mathfrak{S}_m$ gives that 
\begin{equation}\label{eq:entropy-big}
H(\chi)\leq \mathop{\mathbb{E}}_{\sigma \in \mathfrak{S}_m} \sum_{k=1}^m \sum_{e \in \binom{[n]}{2}} \left(\Omega_{k}+\Psi_{k,e}^\sigma+O_m(1/r^3)\right),
\end{equation}
where 
\[\Omega_k=\log(r-2(k-1))+\frac{(k-1)^2}{(r-2m+2)(r-2(k-1))}\quad\text{and}\quad\Psi_{k,e}^\sigma=\frac{\sum_{k'<k}x(\sigma(k),\sigma(k'),e)}{r-k+1}.\]
We compute that
\begin{align*}
\Omega_k &=\log(r-2(k-1))+\frac{(k-1)^2/r}{r-2(k-1)}+O_m(1/r^3)\\
 &=\log(r-2(k-1))+\log\left(1+\frac{(k-1)^2/r}{r-2(k-1)}\right)+O_m(1/r^3)\\
 &=\log(r-2(k-1)+(k-1)^2/r)+O_m(1/r^3)\\
 &=\log r+2\log(1-(k-1)/r)+O_m(1/r^3);
\end{align*}
the crucial point is the identity $r-2(k-1)+(k-1)^2/r=r(1-(k-1)/r)^2$.  Consequently,  
\begin{align*}
\mathop{\mathbb{E}}_{\sigma \in \mathfrak{S}_m} \sum_{k=1}^m \sum_{e \in \binom{[n]}{2}} \Omega_k &=m \binom{n}{2}\log r+2\binom{n}{2} \sum_{k=1}^m \log(1-(k-1)/r)+O_m(n^2/r^3) \\ 
&=m \binom{n}{2}\log r+2 \log(E_{n,m,r})+o_m(1).
\end{align*}
Next, using \eqref{eq:-sum-of-x's}, the formula for the sum of the first $m-1$ squares, and the hypothesis that $n^2/r^3=o(1)$, we find that
\begin{align*}
\mathop{\mathbb{E}}_{\sigma \in \mathfrak{S}_m} \sum_{k=1}^m \sum_{e \in \binom{[n]}{2}}\Psi_{k,e}^\sigma &=\sum_{\ell<\ell'} \sum_e x(\ell,\ell',e) \mathop{\mathbb{E}}_{\sigma \in \mathfrak{S}_m} \frac{1}{r-\max\{\sigma^{-1}(\ell),\sigma^{-1}(\ell')\}+1}\\
 &=\wt(\bpi) \mathop{\mathbb{E}}_{\sigma \in \mathfrak{S}_m} \frac{1}{r-\max\{\sigma^{-1}(1),\sigma^{-1}(2)\}+1}\\
 &=\wt(\bpi) \sum_{j=2}^m \frac{(j-1)/\binom{m}{2}}{r-j+1}\\
 &=\wt(\bpi) \left(\frac{1}{r}+ \sum_{j=2}^m \frac{(j-1)^2/\binom{m}{2}}{r^2} +O_m(1/r^3)\right)\\
 &=\wt(\bpi) \left(\frac{1}{r}+\frac{2m-1}{3r^2} \right)+O_m(n^2/r^3)\\
 &=\frac{\wt(\bpi)}{r-(2m-1)/3} +o_m(1).
\end{align*}
Substituting these bounds back into \eqref{eq:entropy-big}, we conclude that
$$N_{\bpi} \leq (1+o_m(1))r^{m\binom{n}{2}} E_{n,m,r}^{2}e^{\wt(\bpi)/(r-(2m-1)/3)},$$
as desired.
\end{proof}

For convenience, let 
\[
\widehat{r}=r-(2m-1)/3.
\]
If $r$ satisfies \eqref{eq:main2}, then it follows immediately from \cref{lem:entropy-bound} that
\begin{equation}\label{eq:N_pi-bound}
\mathbb{E}[Z_{\boldsymbol{\id}}Z_{\bpi}] \leq (1+o_m(1))E_{n,m,r}^2 e^{\wt(\bpi)/\widehat{r}}
\end{equation}
for each $\bpi \in \mathfrak{S}_n^m$, so
\begin{equation}\label{eq:reduction-to-gamma}
\mathbb{E}[Z^2] \leq (1+o_m(1))E_{n,m,r}^2 n!^m \sum_{\bpi \in \mathfrak{S}_n^{m}} e^{\wt(\bpi)/\widehat{r}}.
\end{equation}
Our goal is to obtain an upper bound on the sum on the right-hand side of \eqref{eq:reduction-to-gamma}.  The following proposition captures the central estimate of the proof.

\begin{proposition}\label{prop:main-estimate}
If $r$ satisfies \eqref{eq:main2},  
then
$$\sum_{\bpi \in \mathfrak{S}_n^{m}} e^{\wt(\bpi)/\widehat{r}}=n!^m(1+o(1)).$$
\end{proposition}

\cref{prop:main-estimate} tells us that if $r$ satisfies \eqref{eq:main2}, then 
$$\mathbb{E}[Z^2]=(1+o_m(1))(E_{n,m,r}n!^m)^2=(1+o_m(1)) \mathbb{E}[Z]^2$$
(using \eqref{eq:first-moment-expression} for the second equality), so $\Var(Z)=o_m(\mathbb{E}[Z])$. Then Chebyshev's Inequality gives that
$$\mathbb{P}[Z=0]=o_m(1),$$
which proves the second part of \cref{thm:main}. Thus, the remainder of this section will be devoted to proving \cref{prop:main-estimate}. Assume in what follows that $r$ satisfies \eqref{eq:main2} or, equivalently, that
\begin{equation}\label{eq:main2_for_hatr}
\widehat r\geq\frac{m\binom{n}{2}}{2\log(n!)}+\frac{m}{2\log n}+\frac{\omega(n)}{(\log n)^2}.
\end{equation}
Each of the $n!^m$ summands on the left-hand side of the equation in \cref{prop:main-estimate} is at least $1$, so we must show that very few summands can be significantly larger than $1$.  To accomplish this, we will split the sum according to the values of $f_{k,k'}(\bpi)$ and $t_{k,k'}(\bpi)$.

Let $L(\bpi)$ be the sequence obtained by listing the pairs $(k,k') \in [m] \times [m]$ with $k<k'$ in decreasing order of $\wt_{k,k'}(\bpi)$ (breaking ties arbitrarily).  Now, let us construct a subsequence ${\vec{p}(\bpi)=(p_1(\bpi), \ldots, p_{m-1}(\bpi))}$ of $L(\bpi)$ recursively as follows. Let $p_1(\bpi)$ be the first pair in $L(\bpi)$. For $2\leq i\leq m-1$, let $p_i(\bpi)$ be the first pair in $L(\bpi)$ such that $p_i(\bpi)\not\in\{p_1(\bpi),\ldots,p_{i-1}(\bpi)\}$ and the (undirected) graph on the vertex set $[m]$ with the edge set $\{p_1(\bpi), \ldots, p_i(\bpi)\}$ is acyclic (where we are identifying the ordered pair $(k,k')$ with the unordered pair $\{k,k'\}$). In other words, $\vec{p}(\bpi)$ is the lexicographically first subsequence of $L(\bpi)$ whose entries form the edges of a spanning tree of the complete graph on $[m]$.  Notice that $\vec{p}(\bpi)$ is uniquely determined by $L(\bpi)$.

Writing $p_\ell(\bpi)=(k_\ell,k_\ell')$, we can use standard rearrangement inequalities to find that 
$$\wt(\bpi) \leq \sum_{\ell=1}^{m-1}\ell\left[\binom{f_\ell(\bpi)}{2}+t_\ell(\bpi) \right],$$
where $f_\ell(\bpi)\coloneq f_{k_\ell, k'_\ell}(\bpi)$ and $t_\ell(\bpi)\coloneq t_{k_\ell, k'_\ell}(\bpi)$.  Note that the number of permutations in $\mathfrak{S}_n$ with $f$ fixed points and $t$ $2$-cycles is at most
$$\binom{n}{f}\binom{n-f}{2t}(2t-1)!!(n-f-2t)!=\frac{n!}{f!2^t t!}.$$
Given a sequence $L$ and tuples $(f_1, \ldots, f_{m-1})$ and $(t_1, \ldots, t_{m-1})$, the number of tuples $\bpi \in \mathfrak{S}_n^m$ satisfying $L(\bpi)=L$ and $f_\ell(\bpi)=f_\ell$ and $t_\ell(\bpi)=t_\ell$ for all $1 \leq \ell \leq m-1$ is at most
$$n!^m \prod_{\ell=1}^{m-1} \frac{1}{f_\ell! 2^{t_\ell} t_\ell!};$$
here, we are crucially using the fact that the entries of $\vec{p}(\bpi)$ form the edge set of an acyclic graph.

Let $\Upsilon$ denote the set of tuples $(f_1,\ldots,f_{m-1},t_1,\ldots,t_{m-1})$ of integers satisfying the following conditions:
\begin{itemize}
\item $0\leq f_1,\ldots,f_{m-1}\leq n$; 
\item $0\leq t_1,\ldots,t_{m-1}\leq n/2$; 
\item $\binom{f_1}{2}+t_1\geq\cdots\geq\binom{f_{m-1}}{2}+t_{m-1}$. 
\end{itemize} 
Let 
\[\Upsilon_{\leq}\coloneq\{(f_1,\ldots,f_{m-1},t_1,\ldots,t_{m-1})\in\Upsilon:{\textstyle\binom{f_1}{2}}+t_1\leq{\textstyle\binom{\log n}{2}}\}\] and \[\Upsilon_{\geq}\coloneq\{(f_1,\ldots,f_{m-1},t_1,\ldots,t_{m-1})\in\Upsilon:{\textstyle\binom{f_1}{2}}+t_1\geq{\textstyle\binom{\log n}{2}}\}.\]
(We remark that in these definitions, the particular choice of $\binom{\log n}{2}$ for the cutoff is not important.)  For each $\bpi\in\mathfrak S_n^m$, the tuple $(f_1(\bpi),\ldots,f_{m-1}(\bpi),t_1(\bpi),\ldots,t_{m-1}(\bpi))$ belongs to $\Upsilon_{\leq}$ or $\Upsilon_{\geq}$.  The tuples $\bpi$ with $(f_1(\bpi),\ldots,f_{m-1}(\bpi),t_1(\bpi),\ldots,t_{m-1}(\bpi))\in\Upsilon_\leq$ will end up contributing the main term of $n!^m(1+o(1))$ in \cref{prop:main-estimate}, while the other tuples will end up contributing only to the $o(n!^m)$ error term. 

Let us begin with $\Upsilon_\leq$.  If $\bpi \in \mathfrak{S}_n^m$ is such that $(f_1(\bpi),\ldots,f_{m-1}(\bpi),t_1(\bpi),\ldots,t_{m-1}(\bpi))\in\Upsilon_\leq$, then $$\wt(\bpi)\leq \binom{m}{2}\binom{\log n}{2}.$$
Since there are at most $n!^m$ such tuples $\bpi$, the sum of $e^{\wt(\bpi)/\widehat r}$ over these tuples is at most
$$n!^m \exp \left( \binom{m}{2}\binom{\log n}{2}\frac{1}{\widehat r} \right) \leq n!^m \exp \left(O_m((\log n)^3/n) \right)=n!^m(1+o(1)),$$
where we have used \eqref{eq:main2}.

We now turn to $\Upsilon_\geq$.  Since there are at most $\binom{m}{2}!$ possibilities for the sequence $L(\bpi)$, the sum of $e^{\wt(\bpi)/\widehat r}$ over all tuples $\bpi \in \mathfrak{S}_n^m$ corresponding to a given tuple $T=(f_1, \ldots, f_{m-1},t_1, \ldots, t_{m-1}) \in \Upsilon_\geq$ is at most $\binom{m}{2}!n!^m X(T)$, where
\begin{equation}\label{eq:X(T)}
X(T)\coloneq\prod_{\ell=1}^{m-1} \left( \frac{1}{f_\ell! 2^{t_\ell} t_\ell!} \exp \left( \left[\binom{f_\ell}{2}+t_\ell\right] \frac{\ell}{\widehat r} \right) \right).
\end{equation}
We wish to find a uniform upper bound on $X(T)$ as $T$ ranges over the elements of $\Upsilon_\geq$.  

We first require a technical lemma. Define $g,h\colon\mathbb R_{\geq 1}\to\mathbb R_{\geq 0}$ by $g(x)=\frac{x(x-1)}{2}$ and ${h(x)=\frac{1+\sqrt{1+8x}}{2}}$ so that $g(h(x))=x$. Let $\Gamma$ denote the gamma function. 

\begin{lemma}\label{lem:Gamma}
If $t$ and $K$ are integers such that $0\leq t\leq K$, then \[\Gamma(h(K-t)+1)2^tt!\geq \Gamma(h(K)+1).\]
\end{lemma}
\begin{proof}
It suffices to prove that $\Gamma(h(K-(t-1))+1)2^{t-1}(t-1)!\leq\Gamma(h(K-t)+1)2^tt!$ whenever $1\leq t\leq K$. The identity $g(h(x))=x$ implies that $h'(x)=\frac{1}{g'(h(x))}=\frac{2}{2h(x)-1}$. Because $h''(x)<0$ for all $x>1$, we have 
\begin{equation}\label{eq:h'}
h(x+1)\leq h(x)+h'(x)=h(x)+\frac{2}{2h(x)-1}.
\end{equation} 
It is routine to verify that 
\begin{equation}\label{eq:Gamma}
\frac{\Gamma(z)}{\Gamma(z+\frac{2}{2z-1})}\geq \frac{1}{2}
\end{equation} for every real number $z\geq 2$. Assume $1\leq t\leq K$. Since $h(K-t)+1\geq h(0)+1=2$, we can set $x=K-t$ in \eqref{eq:h'} and set $z=h(K-t)+1$ in \eqref{eq:Gamma} to find that 
\[\frac{\Gamma(h(K-t)+1)}{\Gamma(h(K-t+1)+1)}\geq\frac{\Gamma(h(K-t)+1)}{\Gamma(h(K-t)+1+\frac{2}{2h(K-t)-1})}\geq\frac{1}{2}.\] Therefore, 
\begin{align*}
\frac{\Gamma(h(K-t)+1)2^tt!}{\Gamma(h(K-(t-1))+1)2^{t-1}(t-1)!}=2t\frac{\Gamma(h(K-t)+1)}{\Gamma(h(K-t+1)+1)}\geq t\geq 1,
\end{align*} as desired. 
\end{proof}

An immediate consequence of \cref{lem:Gamma} is that for each integer $K \geq 0$, the maximum value of \[\frac{1}{f! 2^{t} t!} \exp \left( \left[\binom{f}{2}+t\right] \frac{\ell}{\widehat r} \right),\] over all $f,t \in \mathbb{Z}_{ \geq 0}$ satisfying $\binom{f}{2}+t=K$, is at most
$$\frac{1}{\Gamma(h(K)+1)}\exp \left( \binom{h(K)}{2} \frac{\ell}{\widehat r} \right).$$
Applying this estimate to each multiplicand in the definition of $X(T)$, we find that
\begin{align}
\nonumber \max_{T \in \Upsilon_\geq}X(T) &\leq \max_{\substack{n \geq f_1 \geq \cdots \geq f_{m-1} \geq 0, \\ f_1\geq\log n}} \prod_{\ell=1}^{m-1} \frac{1}{\Gamma(f_\ell+1)} \exp \left( \binom{f_\ell}{2}\cdot \frac{\ell}{\widehat r} \right)\\
 &=\max_{f_1\in[\log n,n]} \exp(\varphi_{1/\widehat r}(f_1)) \max_{f_2\in[0,f_1]} \exp(\varphi_{2/\widehat r}(f_2)) \cdots \max_{f_{m-1} \in [0,f_{m-2}]} \exp(\varphi_{(m-1)/\widehat r}(f_{m-1})),
\end{align}
where the $f_\ell$'s run over intervals of real numbers and we have set
$$\varphi_q(f)\coloneq -\log \Gamma(f+1)+q\binom{f}{2}.$$  
We will now study the behavior of the functions $\varphi_q$.

Since the logarithm of the gamma function is convex (by the Bohr--Mollerup Theorem; see, e.g., {\color{Yellow}\cite{Artin}}) and the function $f\mapsto\binom{f}{2}$ is concave, the function $\varphi_q$ is also concave for all $q \geq 0$.  In particular, the maximum value of $\varphi_q$ on an interval is always assumed at one of the endpoints of the interval.  So the maximum over $f_{m-1}$ is achieved when either $f_{m-1}=0$ or $f_{m-1}=f_{m-2}$.  In the former case, we simply remove this term (since $\varphi_q(0)=0$ for all $q$).  In the latter case, we ``incorporate'' the $f_{m-1}$ term into the preceding $f_{m-2}$ term by noting that
$$\varphi_{(m-2)/\widehat r}(f_{m-2})+\varphi_{(m-1)/\widehat r}(f_{m-2})=2\varphi_{(m-3/2)/\widehat r}(f_{m-2}). $$
We then obtain the same dichotomy for the maximum over $f_{m-2}$, and we continue in this manner until we reach $f_1$, where the maximum occurs when either $f_1=n$ or $f_1=\log n$.  Thus, there is some $1 \leq s \leq m-1$ such that
$$\max_{T \in \Upsilon_\geq}X(T) \leq \max \left\{\exp \left( \sum_{\ell=1}^{s} \varphi_{\ell/\widehat r}(\log n) \right),\,\exp \left( \sum_{\ell=1}^{s} \varphi_{\ell/\widehat r}(n) \right) \right\}.$$
Then the sum of $e^{\wt(\bpi)/\widehat r}$ over all $\bpi \in \mathfrak{S}_n^m$ with ${(f_1(\bpi),\ldots,f_{m-1}(\bpi),t_1(\bpi),\ldots,t_{m-1}(\bpi))\in\Upsilon_{\geq}}$ 
is thus at most
$$\binom{m}{2}! (n^2/2)^{m-1} n!^m \max \left\{\exp \left( \sum_{\ell=1}^{s} \varphi_{\ell/\widehat r}(\log n) \right),\,\exp \left( \sum_{\ell=1}^{s} \varphi_{\ell/\widehat r}(n) \right) \right\},$$
so to prove \cref{prop:main-estimate}, it suffices to show that
$$(m-1)\log(n^2/2)+\max \left\{\sum_{\ell=1}^{s} \varphi_{\ell/\widehat r}(\log n),\,  \sum_{\ell=1}^{s} \varphi_{\ell/\widehat r}(n) \right\} \to -\infty$$
as $n \to \infty$.  We first check that
\begin{align*}
(m-1)\log(n^2/2)+\sum_{\ell=1}^{s} \varphi_{\ell/\widehat r}(\log n) &=(m-1)\log(n^2/2)-s \log( \Gamma(\log n+1))\!+\!\binom{s+1}{2}\binom{\log n}{2} \cdot \frac{1}{\widehat r}\\
 &=-s\log n \log\log n+O_m(\log n)
\end{align*}
tends to $-\infty$ (with room to spare in the asymptotic condition on $\widehat r$ in \eqref{eq:main2_for_hatr}).  We next consider
\begin{equation}\label{eq:whatever}
(m-1)\log(n^2/2)+\sum_{\ell=1}^{s} \varphi_{\ell/\widehat r}(n) =(m-1)\log(n^2/2)-s \log(n!)+\binom{s+1}{2}\binom{n}{2} \cdot \frac{1}{\widehat r}.
\end{equation}
If $s<m-1$, then the right-hand side of \eqref{eq:whatever} is $-(1+o(1))s(1-(s+1)/m)n\log n$, which certainly tends to $-\infty$.  If $s=m-1$, then the right-hand side of \eqref{eq:whatever} becomes
\[
(m-1) \left[\log(n^2/2)-\log(n!)+\frac{m}{2}\binom{n}{2} \cdot \frac{1}{\widehat r} \right],
\]
which tends to $-\infty$ by \eqref{eq:main2_for_hatr}.  This finishes the proof of \cref{prop:main-estimate} and hence also of \cref{thm:main}.

\section{Further Remarks} 

\subsection{Comments on the proof}

When $m=2$, our proof of \cref{thm:main} simplifies considerably but is still nontrivial.  In this case, each graph $G_{\bpi}$ is a disjoint union of edges (corresponding to edges fixed by $\pi_1 \pi_2^{-1}$) and even-length cycles.  One can then bound $N_{\bpi}$ using the known formulas for the chromatic polynomials of cycles in lieu of \cref{lem:entropy-bound}.
Moreover, the proof of \cref{prop:main-estimate} simplifies because the list $L(\bpi)$ consists of the single element $(1,2)$ and it is not necessary to extract the subsequence $\vec{p}(\bpi)$. 

An examination of our proofs shows that \cref{thm:main} continues to hold in the regime where $m$ grows reasonably slowly with $n$.  To optimize this dependence (which we have not attempted), one should tweak some of the parameters appearing in our proof (for instance, the cutoff $\binom{\log n}{2}$ in the definitions of $\Upsilon_\leq, \Upsilon_\geq$); we leave the details to the curious reader.

\subsection{Sharp thresholds}

\cref{thm:main} shows that the existence problem for rainbow stackings exhibits a sharp threshold, in the sense that the transition from having no rainbow stackings with high probability to having rainbow stackings with high probability occurs within an interval of length roughly $(2m-1)/3$.  It is natural to ask if the transition is even sharper; in particular, we pose the following problem.

\begin{problem}
Determine whether or not there exists a function $r_0\colon\mathbb N\to\mathbb R$ such that the following holds. For each $n\geq 1$, let $\chi_1,\ldots, \chi_m\colon \binom{[n]}{2} \to \mathcal{C}_r$ be independent uniformly random $r$-edge-colorings. If $r=r(n)$ satisfies $r(n)<r_0(n)$, then with high probability, there does not exist a rainbow stacking of $\chi_1,\ldots,\chi_m$. If $r=r(n)$ satisfies $r(n)>r_0(n)$, then with high probability, there exists a rainbow stacking of $\chi_1,\ldots,\chi_m$.
\end{problem}

In the past, the second-moment method has often been effective for obtaining analogous sharp results.  A well-known example is the proof 
of the $2$-point concentration of the independence
number of the Erd\H{o}s--R\'enyi random graph $G(n,1/2)$ (see, e.g., {\color{Yellow}\cite{AS,BE,Ma}}). Since, however,
the expected value of this quantity is only $O( \log n)$, such a sharp 
concentration is less dramatic than the sharp transition for rainbow stackings, where the
critical value of $r$ is on the order of $n/\log n$.

\subsection{Rainbow stackings of deterministic edge-colorings} 

It seems interesting to find sufficient (deterministic) conditions for the existence of rainbow stackings, even when $m=2$.  Proper edge-colorings provide a natural starting point.  We note that not every pair of proper edge-colorings has a rainbow stacking.

\begin{proposition}\label{prop:construction}
If $n=2^k-2$ for some integer $k \geq 2$, then there is a pair of proper edge-colorings of $K_n$ with no rainbow stackings.
\end{proposition}

\begin{proof}
We provide an explicit construction of such a pair of colorings, based on a construction described in {\color{Yellow}\cite{AA}} (in the context of transversals in Latin squares).  Let $\mathbb{F}_2^k$ denote the elementary abelian $2$-group of rank $k$.  Let $u_1,v_1,u_2,v_2 \in \mathbb{F}_2^k$ be such that $u_1 \neq v_1$, $u_2 \neq v_2$, and $u_1+v_1=u_2+v_2$.  For each $i\in\{1,2\}$, let us identify the set $\mathbb F_2^k\setminus\{u_i,v_i\}$ with $[n]$ arbitrarily and define the coloring $\chi_i\colon \binom{\mathbb{F}_2^k \setminus\{u_i,v_i\}}{2} \to \mathbb{F}_2^k$ by $\chi_i(\{x,y\})\coloneq x+y$. It is clear that $\chi_1,\chi_2$ are proper edge-colorings.

We will show that the colorings $\chi_1,\chi_2$ do not admit a rainbow stacking.  Consider a bijection ${\sigma \colon \mathbb{F}_2^k \setminus\{u_1,v_1\} \to \mathbb{F}_2^k \setminus\{u_2,v_2\}}$.  We claim that there are distinct elements $x,y \in \mathbb{F}_2^k \setminus\{u_1,v_1\}$ such that $x+\sigma(x)=y+\sigma(y)$.  Indeed, if this were not the case, then the quantities $z+\sigma(z)$ for $z \in \mathbb{F}_2^k \setminus\{u_1,v_1\}$ would all be distinct.  Then, since $\sum_{z \in \mathbb{F}_2^k}z=0$, the quantity $\sum_{z \in \mathbb{F}_2^k \setminus\{u_1,v_1\}}(z+\sigma (z))$ would be the sum of two distinct elements of $\mathbb{F}_2^k$, so it would be nonzero.  At the same time, our choice of $u_1,v_1,u_2,v_2$ ensures that
$$\sum_{z \in \mathbb{F}_2^k \setminus\{u_1,v_1\}}(z+\sigma(z))=-(u_1+v_1)-(u_2+v_2)=0.$$
This contradiction establishes the claim.

Take $x,y$ as in the claim.  The fact that $\chi_1(\{x,y\})=x+y=\sigma(x)+\sigma(y)=\chi_2(\sigma(\{x,y\}))$ shows that $\sigma$ is not a rainbow stacking.
\end{proof}

We remark that the Cayley sum-graph construction in the proof of \cref{prop:construction} does not work when $n$ is sufficiently large and $n\neq 2^k-2$. Indeed, in this case, M\"uyesser and Pokrovskiy showed {\color{Yellow}\cite[Theorem~1.4]{MP}} that for any $n$-element subsets $A$ and $B$ of $\mathbb F_2^k$, there exists a bijection $\sigma\colon A\to B$ such that the sums of the form $a+\sigma(a)$ for $a\in A$ are all distinct.

Motivated by these observations and by a conjecture of Ryser about the
existence of transversals in Latin squares of odd order 
(see {\color{Yellow}\cite{Ry, Mo}}), we ask the following question.

\begin{question}\label{quest:odd}
Is it true that when $n$ is odd, every pair of proper edge-colorings of $K_n$ admits a rainbow stacking?
\end{question}

We remark that the answer to \cref{quest:odd} is ``yes'' when $n=3$ (by inspection) and when $n=5$ (by computer search).  It seems that a general affirmative resolution of this question would be difficult; it may be easier to start with proper edge-colorings in which no color appears a large number of times.

\subsection{Hypergraphs}
It could be interesting to extend our work to random edge-colorings of complete $d$-uniform hypergraphs for $d>2$. 

\section*{Acknowledgments}
Noga Alon was supported by the NSF grant DMS-2154082. Colin Defant was supported by the NSF under grant 2201907 and by a Benjamin Peirce Fellowship at Harvard University.  Noah Kravitz was supported in part  by the NSF Graduate Research Fellowship Program under grant DGE--203965.

{\color{Yellow}
}


\begin{thebibliography}{9}

\bibitem{AA}
S. Akbari and A. Alipour, 
Transversals and multicolored matchings. 
\emph{J. Combin. Des.}, {\bf 12} (2004), 325--332. 

\bibitem{AS}
N. Alon and J. H. Spencer, \emph{The Probabilistic Method, Fourth Edition}. Wiley, 2016.

\bibitem{Artin} 
E. Artin, {\em The Gamma Function}. Holt, Rinehart, and Winston, 1964.

\bibitem{BE} B. Bollob\'as and P. Erd\H{o}s, Cliques in random graphs.  \emph{Math. Proc. Cambridge Philos. Soc.}, {\bf 80} (1976), 419--427. 

\bibitem{Ma} D. W. Matula, \emph{The Largest Clique Size in a Random Graph}.  Technical report, Southern Methodist University, 1976.

\bibitem{Mo}
R. Montgomery,
A proof of the Ryser--Brualdi--Stein conjecture for large even $n$.
\emph{Preprint}, arXiv:2310.19779. 

\bibitem{MP}
A. M\"uyesser and A. Pokrovskiy, A random Hall--Paige conjecture. 
\emph{Preprint}, arXiv:2204.09666. 

\bibitem{Ry}
H. Ryser, 
Neuere probleme der kombinatorik. In \emph{Vortr\"age \"uber Kombinatorik}, 
Oberwolfach, 1967, 69--91. 
\end{thebibliography}
\end{document}